\documentclass[11pt]{article}
\usepackage[margin=1.2in]{geometry}
\usepackage{tabularx}
\usepackage[T1]{fontenc}
\usepackage[utf8]{inputenc}
\usepackage[USenglish, UKenglish, english]{babel}
\usepackage{amssymb}
\usepackage{amsthm}
\usepackage{amsmath}
\usepackage[mathscr]{euscript}
\usepackage{enumitem}
\usepackage{graphicx}
\usepackage{MnSymbol}
 \let\mathscr\relax
\usepackage[scr]{rsfso}
\usepackage{tikz-cd}
\usepackage{tikz}
\usepackage{hyperref}
\usepackage{marginnote}
\usetikzlibrary{arrows.meta}
\usepackage{bm}
\usepackage{float}

\usepackage{cases}

\usepackage{array}

\usepackage{libertine}

\theoremstyle{definition}
\newtheorem{defin}{Definition}[section]
\theoremstyle{definition}

\theoremstyle{plain}
\newtheorem{theo}[defin]{Theorem}
\theoremstyle{plain}
\newtheorem{prop}[defin]{Proposition}
\theoremstyle{plain}
\newtheorem{lem}[defin]{Lemma}
\theoremstyle{plain}

\theoremstyle{definition}

\theoremstyle{definition}

\theoremstyle{definition}

\theoremstyle{plain}
\newtheorem{conj}[defin]{Conjecture}
\theoremstyle{definition}

\theoremstyle{definition}

\theoremstyle{plain}

\theoremstyle{plain}

\theoremstyle{plain}

\theoremstyle{definition}

\theoremstyle{plain}

\theoremstyle{definition}

\theoremstyle{definition}
\newtheorem*{defin*}{Definition}
\theoremstyle{definition}
\newtheorem*{ex*}{Example}
\theoremstyle{plain}
\newtheorem*{theo*}{Theorem}
\theoremstyle{plain}
\theoremstyle{plain}
\newtheorem*{conj*}{Conjecture}
\newtheorem*{prop*}{Proposition}
\theoremstyle{plain}
\newtheorem*{lem*}{Lemma}
\theoremstyle{plain}
\newtheorem*{cor*}{Corollary}
\theoremstyle{definition}
\newtheorem*{rmk*}{Remark}
\theoremstyle{definition}
\newtheorem*{exe*}{Exercise}

\theoremstyle{plain}
\newtheorem{theoA}{Theorem}

\theoremstyle{plain}

\theoremstyle{plain}
\newtheorem{conjA}[theoA]{Question}

\theoremstyle{plain}

\theoremstyle{plain}

\theoremstyle{plain}

\numberwithin{equation}{section}
\usepackage{parskip}

\makeatletter
\def\thm@space@setup{%
  \thm@preskip=\parskip \thm@postskip=0pt
}
\makeatother

\setlist[enumerate]{label=(\roman*)}
\def\uch{{\rm Uch}}

\def\cA{{\mathcal{A}}}
\def\cB{{\mathcal{B}}}
\def\cC{{\mathcal{C}}}
\def\cD{{\mathcal{D}}}
\def\cF{{\mathcal{F}}}

\def\cP{{\mathcal{P}}}

\def\bG{{\mathbf{G}}}

\def\cO{{\mathcal{O}}}
\def\GG{{\mathbb{G}}}

\def\cT{{\mathcal{T}}}

\def\Irr{{\rm Irr}}
\def\Hom{{\rm Hom}}
\def\Uch{{\rm Uch}}

\def\CC{{\mathbb{C}}}
\def\ZZ{{\mathbb{Z}}}
\def\FF{{\mathbb{F}}}
\def\LL{{\mathbb{L}}}
\def\TT{{\mathbb{T}}}
\def\QQ{{\mathbb{Q}}}
\def\irr{{\rm Irr}}

\def\op{{\rm op}}

\def\sd{{\rm sd}}

\def\GL{{\rm GL}}

\def\n{{\mathbf{N}}} %normalizer

\def\c{{\mathbf{C}}} %centralizer
 
\def\z{{\mathbf{Z}}} %center

\def\E{{\mathcal{E}}} %Lusztig series

\def\Deg{{\rm Deg}}

\def\k{{\bf k}}

\def\G{{\mathbf{G}}} %algebraic group

\def\K{{\mathbf{K}}} %algebraic subgroup

\def\L{{\mathbf{L}}} %Levi
\def\cL{{\mathcal{L}}}
\def\S{{\mathbf{S}}} %Torus
\def\cS{{\mathcal{S}}}
\def\CL{{\mathcal{L}}} %couples (chain of Levis)-character

\def\CT{{\mathcal{T}}}

\usepackage{sectsty}
\allsectionsfont{\centering}

\usepackage{xparse,etoolbox}

\newcommand{\blocktheorem}[1]{%
  \csletcs{old#1}{#1}% Store \begin
  \csletcs{endold#1}{end#1}% Store \end
  \RenewDocumentEnvironment{#1}{o}
    {\par\addvspace{1.5ex}
     \noindent\begin{minipage}{\textwidth}
     \IfNoValueTF{##1}
       {\csuse{old#1}}
       {\csuse{old#1}[##1]}}
    {\csuse{endold#1}
     \end{minipage}
     \par\addvspace{1.5ex}}
}

\blocktheorem{theoA}
\blocktheorem{conjA}
\blocktheorem{paraA}
\blocktheorem{corA}

\makeatletter
\def\blfootnote{\gdef\@thefnmark{}\@footnotetext}
\makeatother

\title{
{\huge\bf On $e$-local structures for $\mathbb{Z}_\ell$-spetses}\\
\author{Damiano Rossi and Jason Semeraro}
\date{}
\blfootnote{\emph{$2010$ Mathematical Subject Classification:} $20$C$20$, $55$P$10$, $55$R$35$.
\\
\emph{Key words and phrases:} transporter categories, orbit spaces, $p$-compact groups, spetses, Dade's conjecture.
\\
This work is partially funded by the EPSRC grant EP/W$028794/1$. The first author is supported by the Walter Benjamin Programme of the DFG - Project number 525464727. We are grateful to Gunter Malle for providing several helpful comments on an earlier version of the manuscript.
}
}

\begin{document}

\renewcommand{\thetheoA}{\Alph{theoA}}

\renewcommand{\thepropA}{\Alph{propA}}

\renewcommand{\theconjA}{\Alph{conjA}}

\maketitle

\begin{abstract}
Let $q$ be a prime power, $\ell$ a prime not dividing $q$, and $e$ the order of $q$ modulo $\ell$. We show that the geometric realisation of the  nerve of the transporter category of $e$-split Levi subgroups of a finite reductive group $G$ over $\FF_q$ is homotopy equivalent to the classifying space $BG$ up to $\ell$-completion. We suggest a generalisation of this equivalence to the setting of $\ZZ_\ell$-reflection cosets and establish a related fact involving the associated orbit spaces. We also establish a Dade-like formula for unipotent characters of $\ZZ_\ell$-spetses inspired by a question of Brou\'{e}.
\end{abstract}

\section{Introduction}\label{sec:intro}

For a poset $\cP$, $\Delta(\cP)$ denotes the \textit{order complex} whose simplices are given by totally ordered chains in $\mathcal{P}$. When $G$ is a finite group and $\ell$ is prime, we denote by $\cS_\ell^*(G)$ the poset of non-trivial $\ell$-subgroups of $G$ and refer to the associated simplicial complex $\Delta(\cS_\ell^*(G))$ as the Brown complex. This encodes important information about the $\ell$-local structure of $G$. In the particular case where $G=\G^F$ is a finite reductive group defined over $\mathbb{F}_q$ and $\ell\nmid q$, the first author has recently shown, under suitable assumptions on $\ell$, that there is a $\G^F$-homotopy equivalence 
\begin{equation}
\label{e:rossi}
\Delta(\cS_\ell^\star(G)) \simeq \Delta(\cL_e^*(\G,F))
\end{equation}
where $e:=e_\ell(q)$ is the order of $q$ modulo $\ell$, and $\cL_e^*(\G,F)$ is the poset of (proper) $e$-split Levi subgroups of $G$ (see \cite[Theorem A]{Ros-Homotopy} and \cite[Theorem A]{Ros-Brauer_pairs}).

In our first result, we establish a  version of \eqref{e:rossi} in which the Brown complex is replaced by the classifying space $BG$ of $G$ and $\cL_e^*(\G,F)$ is replaced by the transporter category $\cT_e(\bG,F)$ of $e$-split Levi subgroups with morphisms induced by $\bG^F$-conjugation (see Definition \ref{def:Transporter category for finite reductive groups}). For a small category $\cC$, we denote by $|\cC|$ the geometric realisation of its nerve, called the \textit{classifying space} of $\cC$ (so that $BG \simeq |\cB(G)|$, where $\cB(G)$ is the category with one object and morphism set $G$). We prove the following result.
 
\begin{theoA}
\label{thm:Main, homotopy equivalence}
Let $\bG$ be a   connected reductive group defined over an algebraically closed field of characteristic $p$ and $F:\G \rightarrow \G$ be a Frobenius endomorphism defining an $\mathbb{F}_q$-structure on  $\G$, for some prime power $q$. Let $G=\G^F$, $\ell$ be a very good prime good for $\G$ not dividing $q$, and $e$ be the order of $q$ modulo $\ell$. Then there is a homotopy equivalence
\[BG^\wedge_\ell \simeq |\cT_e(\G,F)|^\wedge_\ell.\]
\end{theoA}

Note that hereinafter, for a space $X$, we write $X^\wedge_\ell$ for the $\ell$-completion of $X$ in the sense of Bousfield--Kan \cite{Bou-Kan72}. The proof of Theorem \ref{thm:Main, homotopy equivalence} is an adaptation of the argument used to prove \cite[Theorem A]{Ros-Homotopy} and relies on Quillen's Theorem A (see Theorem \ref{thm:Quillen theorem A}), the fact that the collection of abelian $\ell$-subgroups of $G$ is \textit{ample} in the sense of Dwyer (see Lemma \ref{lem:Proof of homotopy equivalence}), and a description of $e$-split Levi subgroups as centralisers of certain abelian $\ell$-subgroups (see Lemma \ref{lem:e-split Levi and good primes}).

Recent work of the second author, Kessar and Malle suggests that information about the $\ell$-local structure of $\bG^F$ is determined by the pair $(\GG,q)$ where $\GG$ is a certain $\ZZ_\ell$-reflection coset associated to $(\G,F)$. Here, a \textit{$\ZZ_\ell$-reflection coset} is a pair $\GG=(W\phi,V)$ in which $(W,V)$ is a $\ZZ_\ell$-reflection group and $\phi \in N_{\GL(V)}(W)$. If $\ell$ is a very good prime for such a $\GG$ (see \cite[Definition 2.4]{Kes-Mal-Sem}) with $\phi$ of order prime to $\ell$, $(W,V)$ is simply connected, and $q$ is a prime power not divisible by $\ell$, then to $(\GG,q)$ one may associate, via the theory of $\ell$-compact groups, a topological space $B\GG(q)$ with the property that $B\GG(q) \simeq |B\bG^F|^\wedge_\ell$ whenever $\GG$ is a certain $\ZZ_\ell$-reflection coset associated to $\bG^F$ (see Theorem \ref{t:bm}). On the other hand, work of Brou\'{e}--Malle--Michel \cite{Bro-Mal-Mic99} shows that a $\ZZ_\ell$-reflection coset $\GG$ as above also affords a natural set of $e$-split Levi subcosets on which $W$ acts by conjugation (see Section \ref{sec:Conjecture}). We are thus led to introduce a category $\cT_e(\GG)$ analogous to $\cT_e(\bG,F)$ (see Definition \ref{def:Transporter category for reflection cosets}) about which we pose the following question concerning a possible extension of  Theorem \ref{thm:Main, homotopy equivalence}.

\begin{conjA}
\label{conj:Main, homotopy equivalence}
Let $\mathbb{G}=(W\phi,V)$ be a simply connected $\mathbb{Z}_\ell$-reflection coset, $\ell$ a very good prime for $\mathbb{G}$ not dividing the order of $\phi$, $q$ a prime power not divisible by $\ell$, and $e$ the order of $q$ modulo $\ell$. Then, is it true that $B\GG(q) \simeq \left|\cT_e(\mathbb{G})\right|^\wedge_\ell ?$
\end{conjA}

While we are currently unable to give a complete answer to Question \ref{conj:Main, homotopy equivalence}, we are nonetheless able to establish a connection between the \textit{orbit spaces} (see Section \ref{sec:Orbit space, categories}) of categories whose $\ell$-completed classifying spaces determine the left and right hand sides of the above homotopy equivalence. As recalled in Theorem \ref{t:bm}, to $(\GG,q)$ one may associate a centric linking system $\cL_\ell(\GG(q))$ with the property that $|\cL_\ell(\GG(q))|^\wedge_\ell \simeq B\GG(q)$, and so it makes sense to compare the orbit spaces $\cO(\cL_\ell(\GG(q))$ and $\cO(\cT_e(\mathbb{G}))$. It follows from results of Linckelmann that $\cO(\cL_\ell(\GG(q))$ is contractible (see Theorem \ref{t:linck}). Therefore, our next result offers evidence for an affirmative answer to Question \ref{conj:Main, homotopy equivalence}.

\begin{theoA}
\label{thm:Main contractibility}
Let $\mathbb{G}=(W\phi,V)$ be a simply connected $\mathbb{Z}_\ell$-reflection coset,  $\ell$ a very good prime for $\mathbb{G}$ not dividing the order of $\phi$, $q$ a prime power not divisible by $\ell$, and $e$ the order of $q$ modulo $\ell$. Then $\cO(\cT_e(\mathbb{G}))$ is contractible. In particular, we have
\[\cO(\cL_\ell(\GG(q)) \simeq \cO(\cT_e(\mathbb{G})).\]
\end{theoA}

As an application of the results obtained in \cite{Ros-Homotopy}, the first author has reformulated Dade's Conjecture for finite reductive groups in terms of chains of $e$-split Levi subgroups (see \cite[Theorem E]{Ros-Homotopy} and the statement of \cite[Conjecture C]{Ros24}). This also led to a formulation of Dade's Conjecture for unipotent characters (see \cite{Broue22} and \cite[Theorem B]{Ros-Unip}). We prove a version of the latter result which may be viewed as an analogue of \cite[Theorem B]{Ros-Unip} for a special class of $\ZZ_\ell$-reflection cosets called \textit{$\mathbb{Z}_\ell$-spetses} (see Definition \ref{def:spets}). Attached to a $\ZZ_\ell$-spets is a collection of data which includes analogues of unipotent characters, subdivided into Harish-Chandra series and the like. We write $\k^d_{\rm u}(\GG(q))$  for  the number of unipotent characters with $\ell$-defect $d$ in $\GG(q)$, $\k^d_{\rm u}(\GG(q),\sigma)$ for the number of such characters in the stabiliser of a chain $\sigma$ of $e$-split Levi subcosets of $\GG$, and $\k^d_{\rm u,c}(\mathbb{G}(q))$ for the number of such characters $\chi$ in $\mathbb{G}(q)$ for which $(\GG,\chi)$ is $\Phi_e$-cuspidal (see Definitions \ref{def:Defect in Spetses} and \ref{def:Defect in Spetses, chains}). With these definitions, we can now state our final result.

\begin{theoA}
\label{thm:Main Dade-like}
Let $\mathbb{G}=(W\phi,V)$ be a $\mathbb{Z}_\ell$-spets, $q$ a prime power not divisible by the prime $\ell$, $e$ be the order of $q$ modulo $\ell$ (modulo $4$ if $\ell=2$), and $d$ a non-negative integer. Then
\[\k^d_{\rm u}(\mathbb{G}(q))-\k^d_{\rm u,c}(\mathbb{G}(q))=\sum\limits_{\sigma\in\mathcal{O}\left(\cT_e^\star(\mathbb{G})\right)}(-1)^{|\sigma|+1}\k^d_{\rm u}(\mathbb{G}(q),\sigma)\]
where $\cO(\cT_e^\star(\mathbb{G}))$ is the set of $W$-orbits of (non-trivial) chains of $e$-split Levi subcosets of $\GG$ with fixed largest term $\GG$.
\end{theoA}

By analogy with the situation for finite reductive groups, there is some hope that Theorem \ref{thm:Main Dade-like} might  be related  to \cite[Conjecture 2.12]{KLLS}, an analogue of Dade's conjecture for a fusion system $\cF$, in the special case where $\cF=\cF_\ell(\GG(q))$.

The paper is organised as follows. We start by recalling some relevant definitions from  $\ell$-local group and homotopy theory in Section \ref{sec:llocal}. Next, we recall some key concepts in the theory of finite reductive groups necessary to introduce the transporter system of $e$-split Levi subgroups of such a group. An analogous object is then defined in the context of $\ZZ_\ell$-reflection cosets in Section \ref{sec:Conjecture}. We prove Theorems \ref{thm:Main, homotopy equivalence} and \ref{thm:Main contractibility} in Section \ref{sec:Homotopy equivalence} relying on some standard facts about orbit spaces which are recalled in the Appendix. Finally, in Section \ref{sec:Dade for reflection groups} we introduce some invariants associated to chains of $e$-split Levi subcosets and prove Theorem \ref{thm:Main Dade-like}.

\section{Preliminaries}

\subsection{Fusion, linking and transporter categories}\label{sec:llocal}

We will require the following definition from \cite{Bro-Lev-Oli03}.

\begin{defin}
\label{def:transportercat}
Let $\ell$ be a prime and $G$ be a finite group. The \textit{transporter category} $\CT_\ell(G)$ is the category  whose objects are the $\ell$-subgroups of $G$. Given two $\ell$-subgroups $P$ and $Q$ of $G$, we define
\[\Hom_{\CT_\ell(G)}(P,Q) = \n_G(P,Q):=\left\lbrace x\in G\enspace\middle|\enspace P^x\leq Q\right\rbrace.\]
The morphism in $\Hom_{\CT_\ell(G)}(P,Q)$ corresponding to an element $x \in \n_G(P,Q)$ will be denoted by $\alpha_x$.  More generally, given a collection $\mathcal{S}$ of $\ell$-subgroups of the group $G$ closed under conjugation, we denote by $\CT_\ell^{\mathcal{S}}(G)$ the full subcategory of $\CT_\ell(G)$ whose objects belong to $\mathcal{S}$.
\end{defin}

We have the following result.

\begin{lem}
\label{lem:Proof of homotopy equivalence}
For $G$ and $\ell$ as above, let $\mathcal{A}$ be the set of abelian $\ell$-subgroups of $G$. Then there is a homotopy equivalence
\[BG^\wedge_\ell\simeq \left|\CT_\ell^\mathcal{A}(G)\right|^\wedge_\ell.\]
\end{lem}

\begin{proof}
This follows from \cite[Proposition 1.1, Lemma 1.2, and Lemma 1.3]{Bro-Lev-Oli03}.
\end{proof}

If $S$ is a finite $\ell$-group, a \textit{saturated fusion system} on $S$ is an $EI$ category of $\ell$-subgroups of $S$ satisfying certain axioms which hold whenever $S$ is a Sylow $\ell$-subgroup of a finite group $G$ and $\cF=\cF_S(G)$ is the category with morphism sets
\[\Hom_{\cF_S(G)}(P,Q)=\{c_g: P \rightarrow Q \mid P^g \le Q, g \in G\} \cong \n_G(P,Q)/\c_G(P)\]
for all $P,Q \le S$ (where $c_g$ denotes $G$-conjugation). The reader is referred to \cite{Asc-Kes-Oli} for the precise definition of an abstract saturated fusion system and related concepts. It is a celebrated result of Chermak that to $\cF$ it is always possible to associated another $EI$-category $\cL$ called its \textit{centric linking system}, with the property that $|\cL|^\wedge_\ell \simeq BG^\wedge_\ell$ in the case where $\cF=\cF_S(G)$ as above. For this reason the triple $(S,\cF,\cL)$ is sometimes referred to as an $\ell$-local finite group with classifying space $|\cL|^\wedge_\ell$. We shall need a particular result concerning the \textit{orbit space} $\cO(\cL)$ (see Section \ref{sec:Orbit space, categories}).

\begin{theo}
\label{t:linck}
If $(S,\cF,\cL)$ is an $\ell$-local finite group then $\cO(\cL)$ is contractible.
\end{theo}

\begin{proof}
This follows by combining  \cite[Proposition 1.5]{Lin05} and \cite[Theorem 1.1]{Lin09}.
\end{proof}

\subsection{Finite reductive groups}
\label{sec:red}

Let $\G$ be a connected reductive group defined over an algebraically closed field of characteristic $p$, $F:\G\to \G$ a Frobenius endomorphism endowing the variety $\G$ with an $\mathbb{F}_q$-structure, where char$(\FF_q)=p$, and denote by $\G^F$ the finite reductive group consisting of the $\mathbb{F}_q$-rational points on $\G$. We denote by $P_{(\G,F)}(x)\in\mathbb{Z}[x]$ the order polynomial of $(\G,F)$ as defined in \cite[Definition 1.6.10]{Gec-Mal20}. Given a positive integer $e$, we say that an $F$-stable torus $\S$ of $\G$ is a $\Phi_e$-torus if it satisfies $P_{(\S,F)}(x)=\Phi_e^m$ for some non-negative integer $m$ where $\Phi_e$ denotes the $e$-th cyclotomic polynomial. The centraliser $\c_\G(\S)$ of a $\Phi_e$-torus is called an \textit{$e$-split} Levi subgroup of $(\G,F)$. For $(\G,F)$ as above, the \textit{transporter category} of $e$-split Levi subgroups is defined as follows.

\begin{defin}
\label{def:Transporter category for finite reductive groups}
Define $\cT_e(\G,F)$ to be the category whose objects are the $e$-split Levi subgroups of $(\G,F)$ and with morphism sets given by
\[\Hom_{\cT_e(\G,F)}(\L,\K)=\n_{\G^F}(\L,\K):=\left\lbrace g\in \G^F\enspace\middle|\enspace \L^g\leq \K\right\rbrace\]
for any $e$-split Levi subgroups $\L$ and $\K$ of $(\G,F)$. We will denote by $\beta_g:\L\to \K$ the morphism in $\cT_e(\G,F)$ to which an element $g\in \n_{\G^F}(\L,\K)$ corresponds.
\end{defin}

For some of the arguments presented below, we will need to impose some restrictions on the prime $\ell$.

\begin{defin}
\label{def:Very good primes}
For a finite reductive group $(\G,F)$ defined over $\mathbb{F}_q$, let $\gamma(\G,F)$ be the set of primes $\ell$ such that $\ell$ is odd and good for $\G$, $\ell$ does not divide $q$ nor the index $|\z(\G)^F:\z^\circ(\G)^F|$. Denote by $(\G^*,F^*)$ a pair in duality with $(\G,F)$ and set $\Gamma(\G,F):=(\gamma(\G,F)\cap \gamma(\G^*,F^*))\setminus\{3\}$ if $\G_{\rm ad}^F$ has a component of type $^3\mathbf{D}_4(q^m)$ and $\Gamma(\G,F):=\gamma(\G,F)\cap \gamma(\G^*,F^*)$ otherwise.
\end{defin}

Suppose that $\ell$ does not divide $q$ and denote by $e:=e_\ell(q)$ the (multiplicative) order of $q$ modulo $\ell$. When $\ell$ belongs to the set $\Gamma(\G,F)$ one can establish a strong connection between the $e$-split Levi subgroups of $(\G,F)$ and certain abelian $\ell$-subgroups of $\G^F$. We collect some of these properties in the following lemma.

\begin{lem}
\label{lem:e-split Levi and good primes}
Let $\ell\in\Gamma(\G,F)$. Then the following properties hold.
\begin{enumerate}
\item If $\L$ is an $F$-stable Levi subgroup of $\G$, then $\L$ is $e$-split if and only if $\L=\c_\G^\circ(\z^\circ(\L)_\ell^F)$. Furthermore, this is the case if and only if $\L=\c_\G(\z(\L)_{\Phi_e})$.
\item If $A$ is an abelian $\ell$-subgroup of $\G^F$, then $\c_\G(\z^\circ(\c_\G^\circ(A))_{\Phi_e})$ is an $e$-split Levi subgroup of $(\G,F)$.
\end{enumerate}
\end{lem}

\begin{proof}
The first statement can be found in \cite[Proposition 13.19]{Cab-Eng04} and \cite[Proposition 3.5.5]{Gec-Mal20} while for the second see \cite[Corollary 1.9 (i)]{Ros-Homotopy} 
\end{proof}

\subsection{$\mathbb{Z}_\ell$-reflection cosets}
\label{sec:Conjecture}

In \cite[Section 6]{Kes-Mal-Sem}, the second author, Kessar and Malle adapted several definitions for complex reflection cosets from  \cite{Bro-Mal-Mic99} and \cite{Bro-Mal-Mic14} to the setting of $\ZZ_\ell$-reflection cosets. We include here only those that we require. 

\begin{defin}
A \emph{$\ZZ_\ell$-reflection coset} $\GG$ is a pair $(W\phi,V)$, where:
\begin{enumerate}
\item $V$ is a $\ZZ_\ell$-lattice of finite rank;
\item $W\le\GL(V)$ is a finite reflection group; and
\item $\phi\in\GL(V)$ is an element of finite order normalising~$W$.
\end{enumerate}
\end{defin}

A \emph{reflection subcoset} of $\GG = (W\phi,V)$ is a $\ZZ_\ell$-reflection coset of the form $\GG' = (W'(w\phi)_{|_{V'}},V')$, where $V'$ is a pure sublattice of $V$, $W'$ is a reflection subgroup of $\n_W(V'){|_{V'}}$ (the restriction to $V'$ of the stabiliser $\n_W(V')$ of $V'$), and $w\phi\in W\phi$ stabilises $V'$ and normalises~$W'$. In this situation we write $\GG' \le \GG$ for short. A \emph{toric $\ZZ_\ell$-reflection coset (or torus)} is a $\ZZ_\ell$-reflection coset with trivial reflection group. The \emph{centraliser} of a torus $\TT = ((w\phi)_{|_{V'}},V')$ of $\GG$ is the $\ZZ_\ell$-reflection coset $\c_\GG(\TT):= (\c_W(V')w\phi,V)$ (note that $\c_W(V')$ is a reflection subgroup of $W$ by Steinberg's theorem). Any such $\ZZ_\ell$-reflection coset is called a \emph{Levi} of~$\GG$. Following \cite[Def.~1.44(nc)]{Bro-Mal-Mic14} the \emph{order (polynomial)} of a $\ZZ_\ell$-reflection coset $\GG=(W\phi,V)$ is defined as
\[|\GG|=|\GG|(x):=x^{N(W)}\prod_{i=1}^r \epsilon_i^{-2}(x^{d_i}-\epsilon_i) \in \ZZ_\ell[x]\]
where $N(W)$ is the number of reflections of $W$ and$\{(d_i,\epsilon_i)\mid 1\le i\le r \}$ is the uniquely determined multiset of generalised degrees of $W\phi$ for its action on $V\otimes\QQ_\ell$. 

Let $\zeta\in\ZZ_\ell^\times$ be an $e$-th root of unity (hence of order dividing $\ell-1$) and let $\Phi=x-\zeta\in\ZZ_\ell[x]$. A torus $\TT$ of $\mathbb{G}$ is a \emph{$\Phi$-torus} (or \textit{$e$-torus}) if its polynomial order is a power of $\Phi$. We say that $\TT$ is a \textit{Sylow $\Phi$-torus} (or \textit{Sylow $e$-torus}) if it is maximal among all $\Phi$-tori. The centralisers of $\Phi$-tori of $\GG$ are called \emph{$\Phi$-split} Levi subcosets (or \textit{$e$-split} Levi subcosets), and we write $L_e(\GG)$ for the poset of all such Levi subcosets under the inclusion relation for $\ZZ_\ell$-reflection cosets defined above. If $\mathbb{G}=(W\phi,V)$ is a $\ZZ_\ell$-reflection coset, $\mathbb{H}=(W'\phi',V')$ is a reflection subcoset of $\GG$ and $w\in W$ we set
\[{^w\mathbb{H}}:=({^wW'}\cdot{^w\phi'},wV')=(wW'\phi w^{-1},wV').\] With this definition, we note the following lemma.

\begin{lem}
\label{l:splitpres}
Let $\mathbb{G}=(W\phi,V)$ be a $\ZZ_\ell$-reflection coset and $\LL=\c_\GG(\TT)$ be a $\Phi$-split Levi subcoset of $\GG$ for some $\Phi$-torus $\mathbb{T}$. Then for each  $u\in W$, ${^u\mathbb{L}}$ is a $\Phi$-split Levi. Moreover, $W$ acts transitively on the set $\{\LL=\c_\GG(\TT) \mid \mbox{$\TT$ is a Sylow $\Phi$-torus} \}$.
\end{lem}

\begin{proof}
Writing $\mathbb{T}=((w\phi)_{\mid V'},V')$ and $\mathbb{L}=(\c_W(V')w\phi,V)$, we see that
\[{^u\mathbb{L}}=(\c_W(uV')u(w\phi)w^{-1},V)=\c_\mathbb{G}({^u\mathbb{T}})\]
where ${^u\mathbb{T}}=((uw\phi u^{-1})_{\mid uV'},uV')$ is $\Phi$-toric and of the same rank as $\TT$. We deduce that $^u\LL$ is also a $\Phi$-split Levi subcoset of $\mathbb{G}$. The second statement follows immediately from the Sylow theorem \cite[Theorem 5.1]{Bro-Mal-Mic99}).
\end{proof} 

Lemma \ref{l:splitpres} prompts us to make the following definition.

\begin{defin}
\label{def:Transporter category for reflection cosets}
For a $\ZZ_\ell$-reflection coset $\mathbb{G}=(W\phi,V)$, write $\cT_e(\mathbb{G})$ for the category whose objects are the $e$-split Levi subcosets in $\mathbb{G}$ and with morphisms given by
\[\Hom_{\cT_e(\mathbb{G})}(\mathbb{L},\mathbb{K})=\n_W(\mathbb{L},\mathbb{K}):=\left \lbrace w\in W\enspace\middle|\enspace {^w\mathbb{L}}\leq \mathbb{K}\right\rbrace\]
for every $\mathbb{L}$ and $\mathbb{K}$ in $\cT_e(\GG)$.
\end{defin}

One major advantage that $\ZZ_\ell$-reflection cosets have over $\CC$-reflection cosets is the following result which provides a natural generalisation of $|B\G^F|^\wedge_\ell$ via the theory of $\ell$-compact groups.

\begin{theo}
\label{t:bm}
Let $\GG=(W\phi,V)$ be a simply connected $\mathbb{Z}_\ell$-reflection coset, $\ell$ a very good prime for $\mathbb{G}$ not dividing the order of $\phi$, and $q$ a prime power not divisible by $\ell$. Then  to $(\GG,q)$ one may associate a topological space $B\GG(q)$ with the property that $B\GG(q) \simeq |B\bG^F|^\wedge_\ell$ whenever:
\begin{enumerate}
\item $\bG$ is a connected reductive algebraic group over $\overline{\FF_q}$ with Weyl group $W$ and cocharacter lattice $V_0$;
\item $F:\bG \rightarrow \bG$ is a Frobenius morphism associated to an $\FF_q$-structure corresponding to $\phi \in N_{\GL(V_0)}(W)$; and
\item $\GG=(W\phi,V_0\otimes_\mathbb{Z}\mathbb{Z}_\ell)$.
\end{enumerate}
Moreover $B\GG(q)$ is the classifying space of an $\ell$-local finite group.
\end{theo}

\begin{proof}
The space $B\GG(q)$ is constructed via the theory of $\ell$-compact groups (see \cite[Section 3]{Kes-Mal-Sem}). In particular, see \cite[Remark 3.3(a)]{Kes-Mal-Sem}) for the stated connection with $\bG^F$ and \cite[Remark 3.3(b)]{Kes-Mal-Sem}) for the existence of an associated $\ell$-local finite group.
\end{proof}

\section{Proofs of Theorems \ref{thm:Main, homotopy equivalence} and \ref{thm:Main contractibility}}
\label{sec:Homotopy equivalence}

We first prove Theorem \ref{thm:Main, homotopy equivalence}. Our proof relies on the terminology and results introduced in Sections \ref{sec:llocal} and \ref{sec:red}, together with the notion of a filtered category (see Definition \ref{def:filtered}) which we use in our application of Quillen's Theorem A. Recall from the introduction that $|\cC|$ denotes the classifying space of a small category $\cC$.

\begin{proof}[Proof of Theorem \ref{thm:Main, homotopy equivalence}]
Let $\G^F$ be a finite reductive group and $\ell\in\Gamma(\G,F)$. By Lemma \ref{lem:Proof of homotopy equivalence} it suffices to  show that there is  a homotopy equivalence of classifying spaces
\[|\cT^\cA_\ell(\G^F)| \simeq |\cT_e(\G,F)|.\]
First,  consider the functor
\[f:\CT_\ell^{\mathcal{A}}(\G^F)^\op\to \cT_e(\G,F)\]
given by sending
\begin{enumerate}
\item an object $A$ of $\cT_\ell^\cA(\bG^F)$ to $f(A):= \c_\G(\z^\circ(\c_\G^\circ(A))_{\Phi_e})$;
\item a morphism $\alpha_x^\op:A\to A'$ in $\CT_\ell^\mathcal{A}(\G^F)^\op$ to $f(\alpha_x^\op):=\beta_{x^{-1}}:f(A)\to f(A').$
\end{enumerate}
We claim that $f$ is well-defined (as a functor). Indeed, by Lemma \ref{lem:e-split Levi and good primes} (ii), $f(A)$ is an object of $\cT_e(\G,F)$ for every object $A$ of $\CT_\ell^\mathcal{A}(\G^F)^{\rm op}$. Suppose that $A,A',A''$ are objects in $\CT_\ell^\mathcal{A}(\G^F)^\op$ and that $\alpha_x^\op:A\to A'$ and $\alpha_y^\op:A'\to A''$ ($x,y\in \G^F$) are morphisms between them. Then we have that $A'^x\leq A$ so $f(A)^{x^{-1}}\leq f(A')$ and $\beta_{x^{-1}}:f(A)\to f(A')$  lies in $\cT_e(\G,F)$. Moreover, $\alpha_y^\op\circ\alpha_x^\op=\alpha_{yx}^\op$ and $\beta_{x^{-1}y^{-1}}=\beta_{y^{-1}}\circ\beta_{x^{-1}}$, so $f(\alpha_y^\op\circ\alpha_x^\op)=f(\alpha_y^\op)\circ f(\alpha_x^\op)$, as needed.

We next show, using Theorem \ref{thm:Quillen theorem A}, that $f$ induces a homotopy equivalence
\begin{equation}
\label{e:fhtpyeq}
\left|\CT_\ell^{\mathcal{A}}\left(\G^F\right)^\op\right|\simeq \left|\cT_e(\G,F)\right|.
\end{equation}
To this end, by using Lemma \ref{lem:Filtered categories are contractible} below, it suffices to show that the comma category  $f/\L$ is filtered  for each $e$-split Levi subgroup $\L$ of $(\G,F)$. Set $Z:=\z^\circ(\L)^F_\ell$. Then it follows from Lemma \ref{lem:e-split Levi and good primes} (i), that $f(Z)=\L$ and we thus obtain an object $\beta_1:f(Z)\to \L$ in $f/\L$ with the the property that, for any objects  $\beta_x:f(A)\to \L$ and $\beta_y:f(A')\to \L$ in  $f/\L$, there are morphisms from $\beta_x:f(A)\to \L$ and $\beta_y:f(A')\to \L$ to $\beta_1:f(Z)\to \L$ given respectively by the commutative diagrams
\begin{center}
\begin{tikzcd}
f(A)\arrow[rd, "\beta_x"]\arrow[dd, swap, "f(\alpha_{x^{-1}})=\beta_x"] &&& f(A')\arrow[rd, "\beta_y"]\arrow[dd, swap, "f(\alpha_{y^{-1}})=\beta_y"]
\\
&\L &&& \L.
\\
f(Z)\arrow[ru, swap, "\beta_1"] &&& f(Z)\arrow[ru, swap, "\beta_1"]

\end{tikzcd}
\end{center}
It remains to check that Definition \ref{def:filtered}(ii) holds. Consider two parallel morphisms in $f/\L$ between the objects $\beta_x:f(A)\to \L$ and $\beta_y:f(A')\to \L$ given by the following commutative diagrams
\begin{center}
\begin{tikzcd}
f(A)\arrow[rd, "\beta_x"]\arrow[dd, swap, "f(\alpha_g)=\beta_{g^{-1}}"] &&& f(A)\arrow[rd, "\beta_x"]\arrow[dd, swap, "f(\alpha_{h})=\beta_{h^{-1}}"]
\\
&\L &&& \L
\\
f(A')\arrow[ru, swap, "\beta_y"] &&& f(A')\arrow[ru, swap, "\beta_y"]
\end{tikzcd}
\end{center}
The commutativity of the left-hand diagram implies that $x=g^{-1}y$ while the commutativity of the right-hand diagram implies that $x=h^{-1}y$. Considering the object $\beta_1:f(Z)\to \L$ of $f/\L$ and the morphism in $f/\L$ from the object $\beta_y:f(A')\to \L$ to $\beta_1:f(Z)\to \L$ given by the diagram
\begin{center}
\begin{tikzcd}
f(A')\arrow[rd, "\beta_y"]\arrow[dd, swap, "f(\alpha_{y^{-1}})=\beta_y"]
\\
&\L
\\
f(Z)\arrow[ru, swap, "\beta_1"]
\end{tikzcd}
\end{center}
we see that the compositions
\begin{center}
\begin{tikzcd}
f(A)\arrow[rd, "\beta_x"]\arrow[d, swap, "f(\alpha_g)=\beta_{g^{-1}}"] &&& f(A)\arrow[rd, "\beta_x"]\arrow[d, swap, "f(\alpha_{h})=\beta_{h^{-1}}"]
\\
f(A')\arrow[r, swap, "\beta_y"]\arrow[d, swap, "f(\alpha_{y^{-1}})=\beta_y"] &\L && f(A')\arrow[r, swap, "\beta_y"]\arrow[d, swap, "f(\alpha_{y^{-1}})=\beta_y"] &\L
\\
f(Z)\arrow[ru, swap, "\beta_1"] &&& f(Z)\arrow[ru, swap, "\beta_1"]
\end{tikzcd}
\end{center}
coincide since $\beta_y\circ\beta_{g^{-1}}=\beta_x=\beta_y\circ\beta_{h^{-1}}$. We conclude that \eqref{e:fhtpyeq} holds.

Finally,  note that by \cite[Exercise 11.2.9]{Ric20} the classifying space of the opposite category $\CT_\ell^{\mathcal{A}}(\G^F)^\op$ is homeomorphic to that of $\CT_\ell^{\mathcal{A}}(\G^F)$, and the result follows.
\end{proof}

We next prove Theorem \ref{thm:Main contractibility}. For terminology and results concerning orbit spaces of simplicial complexes and categories we shall need, the reader is referred to Sections \ref{sec:Orbit space, simplicial complex} and \ref{sec:Orbit space, categories} respectively. We require the following Lemma which relates the orbit space of the category $\cT_e(\GG)$ to the orbit space of the simplicial complex $\Delta(L_e(\mathbb{G}))$ defined in Section \ref{sec:Conjecture}. The proof is inspired by that of \cite[Proposition 3.4]{Lin05}.

\begin{lem}
\label{lem:Isomorphism of posets}
The orbit spaces $\mathcal{O}(\cT_e(\mathbb{G}))$ and $\mathcal{O}_W(\Delta(L_{e}(\mathbb{G})))$ are homotopy equivalent.
\end{lem}

\begin{proof}
We first prove that $[S(\cT_e(\mathbb{G}))]$ and $S(\Delta(L_e(\mathbb{G})))/W$ are isomorphic as posets. We follow the argument given in the proof of \cite[Proposition 3.4]{Lin05}. First, note that any element in the poset $S(\Delta(L_e(\mathbb{G})))$ is a chain $\mathbb{L}_0<\mathbb{L}_1<\dots <\mathbb{L}_n$ of $e$-split Levi reflection subcosets of $\mathbb{G}$ and hence an object in $S(\cT_e(\mathbb{G}))$. Since $W$-conjugate chains in $S(\Delta(L_e(\mathbb{G})))$ are isomorphic in $S(\cT_e(\mathbb{G}))$, we thus obtain a well-defined natural map
\[\Theta: S(\Delta(L_e(\mathbb{G})))/W\to [S(\cT_e(\mathbb{G}))]\]
of posets. We first prove that $\Theta$ is surjective. If
\[\begin{tikzcd}\sigma:\mathbb{K}_0\arrow[r, "\varphi_0"] &\mathbb{K}_1\arrow[r, "\varphi_1"]&\dots\arrow[r, "\varphi_{n-1}"] &\mathbb{K}_n \end{tikzcd}\]
is a chain of non-isomorphisms in $S(\cT_e(\mathbb{G}))$ then $\sigma$ is isomorphic in $S(\cT_e(\mathbb{G}))$ to the chain
\[\tau:\mathbb{L}_0<\mathbb{L}_1<\dots<\mathbb{L}_n \mbox{ where }\mathbb{L}_n=\mathbb{K}_n \mbox{ and } \mathbb{L}_i=\varphi_{n-1}\circ\dots\circ \varphi_i(\mathbb{K}_i)\]
via the isomorphism $\mu$ given by the family of morphisms $\mu_n:={\rm Id}_{\mathbb{K}_n}$ and $\mu_i:=\varphi_{n-1}\circ\dots\circ \varphi_i$ for each $i < n$. It follows that $[\sigma]=[\tau]$ lies in the image of $\Theta$. 

Moreover $\Theta$ is injective since if $\mathbb{L}_0<\mathbb{L}_1<\dots<\mathbb{L}_n$ and $\mathbb{K}_0<\mathbb{K}_1<\dots<\mathbb{K}_n$ are chains of $e$-split reflection subcosets of $\mathbb{G}$ which are isomorphic as objects in $S(\cT_e(\mathbb{G}))$ then there is an element $w\in W$ such that $\mathbb{K}_i={^w\mathbb{L}}_i$ for every $i=0,\dots, n$. 

We have thus shown that the posets $(\Delta(L_e(\mathbb{G})))/W$ and $[S(\cT_e(\mathbb{G}))]$ are isomorphic. Now, by taking geometric realisations we get
\begin{equation}
\label{eq:Isomorphism of posets, 1}
\mathcal{O}(\cT_e(\mathbb{G})):=|\Delta([S(\cT_e(\mathbb{G}))])|\simeq |\Delta(S(\Delta(L_e(\mathbb{G})))/W)|
\end{equation}
while by \cite[Proposition 3.14]{Bab-Koz05} we know that
\begin{equation}
\label{eq:Isomorphism of posets, 2}
|\Delta(S(\Delta(L_e(\mathbb{G})))/W)|\simeq |\Delta(S(\Delta(L_e(\mathbb{G}))))|/W =:\mathcal{O}_W(\Delta(L_e(\mathbb{G}))).
\end{equation}
The result now follows by combining \eqref{eq:Isomorphism of posets, 1} and \eqref{eq:Isomorphism of posets, 2}.
\end{proof}

\begin{proof}[Proof of Theorem \ref{thm:Main contractibility}]
By Lemma \ref{lem:Isomorphism of posets}, to prove that $\mathcal{O}\left(\cT_e(\mathbb{G})\right)$ is contractible, it suffices to prove that $\mathcal{O}_W(\Delta(L_e(\mathbb{G})))$ is contractible. To this end, let $\Gamma$ be the directed graph with vertex set $L_e(\mathbb{G})$ and an edge $\mathbb{L}\to\mathbb{K}$ if $\mathbb{L} \le \mathbb{K}$ in $L_e(\mathbb{G})$. Observe that $\Gamma$ is a directed $W$-graph, that is, a directed graph on which $W$ acts via automorphisms of directed graphs. Then, by definition, the flag complex $X(\Gamma)$ (see Definition \ref{d:flag}) coincides with the order complex $\Delta(L_e(\mathbb{G}))$. Since $\Gamma$ is well-founded (see Definition \ref{d:morse}(iii)), it suffices to show that the hypotheses of Lemma \ref{lem:Morse theory} hold for the $W$-graph $\Gamma$.  First, observe that the minimal elements of $\Gamma$ are exactly the centralisers of Sylow $\Phi_e$-tori in $\GG$ and are hence all $W$-conjugate by Lemma \ref{l:splitpres}. Thus Lemma \ref{lem:Morse theory}(i) holds. Secondly, let $\sigma$ be a non-minimal simplex  in $X(\Gamma)$ and $\mathbb{L} \in \sigma$ be its minimal element. Then,  $\Gamma_{\downarrow}^\sigma=\Gamma_{\downarrow}^{\mathbb{L}}$ (see \cite[Remark 11]{Bux99}) and the set of vertices of $\Gamma_{\downarrow}^{\mathbb{L}}$ coincides with $L_e(\mathbb{L})$. Write $\mathbb{L}=(W_{\mathbb{L}}\phi',V)$. By applying Lemma \ref{l:splitpres} to the reflection subcoset $\mathbb{L}$, we see that $W_\mathbb{L}$ acts transitively on the set of minimal elements of $\Gamma_{\downarrow}^{\mathbb{L}}$. Since $W_\mathbb{L} \leq W_\sigma$, this shows that $W_\sigma$ acts transitively on the minimal elements of $\Gamma_{\downarrow}^\mathbb{\sigma}$ and Lemma \ref{lem:Morse theory}(ii) holds. The result follows.
\end{proof}

\section{An alternating sum for $\mathbb{Z}_\ell$-spetses}
\label{sec:Dade for reflection groups}

To start, we recall the statement of Dade's Conjecture in its block-free form (see \cite[Conjecture 9.25]{Nav18} or, more generally, \cite[Conjecture 2.5.4]{Ros-Thesis}). Let $G$ be a finite group and $\chi\in\irr(G)$. Recall that the $\ell$-defect of $\chi$ is defined as $d_\ell(\chi):=\nu_\ell(|G|)-\nu_\ell(\chi(1))$, where $\nu_\ell$ denotes the $\ell$-adic valuation. For a non-negative integer $d$, let $\irr^d(G)$ be the set of irreducible characters $\chi\in\irr(G)$ with $d_\ell(\chi)=d$. Consider the orbit space $\mathcal{O}_G(\mathcal{S}_\ell(G))$ of chains of $\ell$-subgroups $\sigma=\{P_0=1<P_1<\dots<P_n\}$ and for each $\sigma\in \mathcal{O}_G(\mathcal{S}_\ell(G))$, denote by $|\sigma|:=n$ its length and by $G_\sigma$ the stabiliser of $\sigma$ in $G$. Furthermore, let $\k^d(G_\sigma)$ be the number of characters in $\irr^d(G_\sigma)$.

\begin{conj}[Dade's Conjecture]
\label{conj:dade}
For any finite group $G$ and any positive integer $d$, we have
\[\sum\limits_{\sigma\in\mathcal{O}_G(\mathcal{S}_\ell(G))}(-1)^{|\sigma|}\k^d(G_\sigma)=0.\]
\end{conj}

When $G=\G^F$ is a finite reductive group defined over $\mathbb{F}_q$ and $\ell$ does not divide $q$, work of the first named author shows, under suitable restrictions, that Dade's Conjecture can be reformulated in terms of $e$-Harish-Chandra theory with $e$ the order of $q$ modulo $\ell$ (see \cite[Conjecture C]{Ros24}). For this, one has to replace the set of $\ell$-chains appearing in Dade's Conjecture with chains of $e$-split Levi subgroups (see \cite[Theorem E]{Ros-Homotopy} and \cite[Proposition 7.10]{Ros24}). Moreover, the restriction $d>0$ in Conjecture \ref{conj:dade}, may be removed by introducing a corrective term involving $e$-cuspidal characters to the alternative sum (see \cite[Proposition 4.18]{Ros24} and \cite[Lemma 4.8]{Ros-Homotopy}).  Following a question raised by  Michel Brou\'e \cite{Broue22}, the first named author in \cite{Ros-Unip} considered a restricted form of this alternating sum involving just the unipotent characters. More precisely, under the hypothesis of \cite[Theorem B]{Ros-Unip}, we have
\begin{equation}
\label{eq:Dade for unipotent, group case}
\k^d_{\rm u}(\G^F)-\k^d_{\rm u, c}(\G^F)=\sum\limits_{\sigma\in\mathcal{O}_{\G^F}(\CL_e^\star(\G,F))}(-1)^{|\sigma|}\k^d_{\rm u}(\G^F_\sigma)
\end{equation}
for every non-negative integer $d$ and where
\begin{enumerate}
\item $\k^d_{\rm u}(\G^F)$ is the number of unipotent characters of $\ell$-defect $d$ of $\G^F$;
\item $\k^d_{\rm u, c}(\G^F)$ is the number of $e$-cuspidal unipotent characters of $\ell$-defect $d$ of $\G^F$;
\item $\mathcal{O}_{\G^F}(\CL_e^\star(\G,F))$ is the set of conjugacy classes of $e$-chains $\sigma\neq \{\G\}$; and
\item $\k^d_{\rm u}(\G_\sigma^F)$ is the number of irreducible characters of $\ell$-defect $d$ in $\G_\sigma^F$ that are unipotent in the sense of \cite[(5.10)]{Ros-Unip}, for every $\sigma\in\mathcal{O}_{\G^F}(\CL_e^\star(\G,F))$.
\end{enumerate}

The aim of this section is to prove an analogue of \eqref{eq:Dade for unipotent, group case} for $\mathbb{Z}_\ell$-spetses. We recall the following definition from \cite[Definition 6.1]{Kes-Mal-Sem}.

\begin{defin}
\label{def:spets}
A \textit{$\ZZ_\ell$-spets} is a $\ZZ_\ell$-reflection coset $(W\phi,V)$ such that the extension of $(W,V)$ to $\mathbb{C}$ is spetsial as defined in \cite[Section 8]{Mal00}.
\end{defin}

As in \cite{Bro-Mal-Mic14} there exists a set $\uch(\mathbb{G})$ of \textit{unipotent characters} together with a degree polynomial $\Deg(\rho)$ for every $\rho\in \uch(\mathbb{G})$. Moreover, by \cite[4.31 (1)]{Bro-Mal-Mic14}, there is a partition
\begin{equation}
\label{eq:Partition e-HC}
\uch(\mathbb{G})=\coprod\limits_{(\mathbb{L},\lambda)}\uch_\mathbb{G}(\mathbb{L},\lambda)
\end{equation}
where $(\mathbb{L},\lambda)$ runs over a set of representatives for the action of $W$ on the \textit{$\Phi$-cuspidal pairs} of $\mathbb{G}$ (see \cite[Definition 4.29]{Bro-Mal-Mic14}). 

\begin{defin}
\label{def:Defect in Spetses}
Let $q$ be a prime power and $\ell$ a prime number not dividing $q$. For every $\chi\in\uch(\mathbb{G})$ we define the \textit{$(q,\ell)$-defect} of $\chi\in\uch(\mathbb{G})$ as
\[d_{q,\ell}(\chi):=\nu_\ell\left(|\mathbb{G}|_{x=q}\right)-\nu_\ell\left(\Deg(\chi)_{x=q}\right).\]
We then set
\[\k^d_{\rm u}(\mathbb{G}(q)):=\left|\left\lbrace \chi\in\uch(\mathbb{G}) \enspace\middle|\enspace d_{q,\ell}(\chi)=d\right\rbrace\right|\]
If $\Phi=x-\zeta\in\mathbb{Z}_\ell[x]$, where $\zeta\in\mathbb{Z}_\ell^\times$ is an $e$-th root of unity and $e$ is the order of $q$ modulo $\ell$ (modulo $4$ if $\ell=2$), we set
\[\k^d_{\rm u}(\mathbb{G}(q),(\mathbb{L},\lambda)):=\left|\left\lbrace \chi\in\uch_\mathbb{G}(\mathbb{L},\lambda) \enspace\middle|\enspace d_{q,\ell}(\chi)=d\right\rbrace\right|\]
for every $\Phi$-cuspidal pair $(\mathbb{L},\lambda)$ of $\mathbb{G}$. Finally, motivated by \eqref{eq:Dade for unipotent, group case}, we define $\k^d_{\rm u,c}(\mathbb{G}(q))$ to be the number of characters $\chi\in\uch(\mathbb{G})$ such that $d_{q,\ell}(\chi)=d$ and $(\mathbb{G},\chi)$ is a $\Phi$-cuspidal pair.
\end{defin}

The invariants defined above can be controlled in terms of relative Weyl groups. For this, we need a generalisation of \cite[Proposition 6.5]{Kes-Mal-Sem} to non-principal series. Recall below that the relative Weyl group $W_\mathbb{G}(\LL,\lambda)$ is a finite group and we can therefore consider the set $\irr^d(W_\mathbb{G}(\mathbb{L},\lambda))$ of its irreducible character of fixed $\ell$-defect $d\geq 0$.

\begin{prop}
\label{prop:e-HC bijection}
Let $\mathbb{G}=(W\phi,V)$ be a $\mathbb{Z}_\ell$-spets, $q$ a prime power not divisible by $\ell$ and $e$ the order of $q$ modulo $\ell$ (modulo $4$ if $\ell=2$). Let $\zeta\in\mathbb{Z}_\ell^\times$ be an $e$-th root of unity and set $\Phi=x-\zeta\in\mathbb{Z}_\ell[x]$. Then, for every $\Phi$-cuspidal pair $(\mathbb{L},\lambda)$ of $\mathbb{G}$ there exists a bijection
\[\Psi_{(\mathbb{L},\lambda)}^\mathbb{G}:\uch_\mathbb{G}(\mathbb{L},\lambda) \to \irr\left(W_\mathbb{G}(\mathbb{L},\lambda)\right)\]
such that $d_{q,\ell}(\chi)-d_{q,\ell}(\lambda)=d_\ell\left(\Psi_{(\mathbb{L},\lambda)}^\mathbb{G}(\chi)\right)$
for every $\chi\in\uch_\mathbb{G}(\mathbb{L},\lambda)$. Consequently, for every non-negative integer $d$, we have
\[\k^d_{\rm u}(\mathbb{G}(q),(\mathbb{L},\lambda))=\left|\irr^{d-d_{q,\ell}(\lambda)}(W_\mathbb{G}(\mathbb{L},\lambda))\right|.\]
\end{prop}

\begin{proof}
Observe that \cite[Axiom 4.31]{Bro-Mal-Mic14} yields a bijection
\[\Psi_{(\LL,\lambda)}^\GG:\Uch_\GG(\LL,\lambda) \rightarrow \Irr(W_\GG(\LL,\lambda)\]
as above. Furthermore, by \cite[Axiom 4.31(2(b))]{Bro-Mal-Mic14}, for each $\chi \in \Uch_\GG(\LL,\lambda)$ we have
\[\nu_\ell(\Deg(\chi)_{x=q})=\nu_\ell\left(\frac{|\GG|_{x=q} \cdot \Deg(\lambda)_{x=q} \cdot \Psi(\chi)(1)}{|W_\GG(\LL,\lambda)|\cdot |\LL|_{x=q}}\right)\]
from which we get 
\begin{align*}
d_{q,\ell}(\chi)-d_{q,\ell}(\lambda)&=\nu_\ell(|\GG|_{x=q})-\nu_\ell(\Deg(\chi)_{x=q})-\nu_\ell(|\LL|_{x=q})+\nu_\ell(\Deg(\lambda)_{x=q})
%\\
%&=\nu_\ell(|\GG|_{q})-\nu_\ell(|\GG|_{q})-\nu_\ell(\chi(1)_{q})-\nu_\ell(\Psi(\chi)(1))+\nu_\ell(|W_\GG(\LL,\lambda)|)+\nu_\ell(|\LL|_q)\nu_\ell(|\LL|_q)-\nu_\ell(\lambda(1)_q)
\\
&=\nu_\ell(|W_\GG(\LL,\lambda)|)-\nu_\ell(\Psi^\GG_{(\LL,\lambda)}(\chi)(1))
\\
&=d_\ell\left(\Psi_{(\mathbb{L},\lambda)}^\mathbb{G}(\chi)\right)
\end{align*}
as needed.
\end{proof}

Our next aim is to define an analogue of the invariants $\k^d_{\rm u}(\G^F_\sigma)$ from \eqref{eq:Dade for unipotent, group case} for $\mathbb{Z}_\ell$-spetses. For the rest of this section we consider $\mathbb{G}=(W\phi,V)$, $q$, $\ell$, $e$, $\zeta$, and $\Phi$ as in the statement above. Let $\sigma$ be a chain of $e$-split Levi subcosets in $\mathbb{G}$. Since we will only consider such chains up to $W$-conjugation, we can think of $\sigma$ as an object in the orbit category $\mathcal{O}(\cT_e(\mathbb{G}))$ (see the proof of Lemma \ref{lem:Isomorphism of posets}).

\begin{defin}
\label{def:Relative Weyl group, chains}
Let $\mathbb{L}(\sigma)$ be the minimal term in $\sigma$ and consider a $\Phi$-cuspidal pair $(\mathbb{L},\lambda)$ of $\mathbb{L}(\sigma)$. Let $\n_W(\sigma)$ be the intersection of the normalisers $\n_W(\mathbb{K})$ for $\mathbb{K}$ running over the terms of $\sigma$ and set $W_\mathbb{G}(\sigma,\mathbb{L}):=(\n_W(\mathbb{L})\cap \n_W(\sigma))/W_\mathbb{L}$. Since $W_\mathbb{G}(\sigma,\mathbb{L})\leq W_\mathbb{G}(\mathbb{L})$, it follows from \cite[4.3-4.4]{Bro-Mal-Mic14} that $W_\mathbb{G}(\sigma,\mathbb{L})$ acts on the set $\uch(\mathbb{L})$. We denote by $W_\mathbb{G}(\sigma,(\mathbb{L},\lambda))$ the stabiliser of $\lambda$ in $W_\mathbb{G}(\sigma,\mathbb{L})$.
\end{defin}

Now, inspired by \cite[Definition 5.2]{Ros-Unip} (see also \cite[Lemma 5.5]{Ros24}) we give the following definition.

\begin{defin}
\label{def:Defect in Spetses, chains}
Let $\sigma\in\mathcal{O}(\cT_e(\mathbb{G}))$ with minimal term $\LL(\sigma)$ and consider a $\Phi$-cuspidal pair $(\mathbb{L},\lambda)$ of $\mathbb{L}(\sigma)$. For every non-negative integer $d$ we define
\[\k^d_{\rm u}(\mathbb{G}(q),\sigma, (\mathbb{L},\lambda)):=\left|\irr^{d-d_{q,\ell}(\lambda)}\left(W_\mathbb{G}(\sigma,\mathbb{L},\lambda)\right)\right|.\]
Moreover, set
\[\k^d_{\rm u}(\mathbb{G}(q),\sigma):=\sum\limits_{(\mathbb{L},\lambda)}\k_{\rm u}^d\left(\mathbb{G}(q), \sigma,(\mathbb{L},\lambda)\right)\]
where $(\mathbb{L},\lambda)$ runs over  a set of representatives for the action of $W_{\mathbb{L}(\sigma)}$ on the set of $\Phi$-cuspidal pairs of $\mathbb{L}(\sigma)$.
\end{defin}

\begin{defin}
\label{def:Orbits space for alternating sum}
Below, we denote by $\mathcal{O}(\cT_e^\star(\mathbb{G}))$ the set of $W$-orbits of chains $\sigma$ with fixed largest term $\mathbb{G}$ and satisfying $\sigma\neq \{\mathbb{G}\}$. For any such chain $\sigma=\{\mathbb{L}_0=\mathbb{L}(\sigma)<\dots <\mathbb{L}_n=\mathbb{G}\}$, we write $|\sigma|:=n$ to denote the \textit{length} of $\sigma$.
\end{defin}

In the following proposition, we show that the invariant $\k^d_{\rm u}(\mathbb{G}(q),\sigma)$ from Definition \ref{def:Defect in Spetses, chains} agrees in the group case with $\k_{\rm u}^d(\G_\sigma^F)$ from \eqref{eq:Dade for unipotent, group case}.

\begin{prop}
Let $\G^F$ be a finite reductive group defined over a field $\mathbb{F}_q$ and let $\ell$ be an odd prime god for $\G$ and not dividing $q$, $e$ the order of $q$ modulo $\ell$, and consider a chain $\sigma$ of $e$-split Levi subgroups with smallest term $\L(\sigma)$. Suppose that \cite[Conjectured Parametrisation B]{Ros24} holds for $(\G,F)$ with respect to any unipotent $e$-cuspidal pair (see, for instance, \cite[Theorem C]{Ros-Unip} for the case of groups of types $\bf{A}$, $\bf{B}$, and $\bf{C}$). Then, for every non-negative integer $d$ and any $\sigma$, we have
\[\k^d_{\rm u}\left(\G^F_\sigma\right)=\sum\limits_{(\L,\lambda)}\left|\irr^{d-d(\lambda)}\left(W_\G(\sigma,(\L,\lambda))^F\right)\right|\]
where $(\L,\lambda)$ runs over a set of representatives for the $\L(\sigma)^F$-conjugacy classes of unipotent $e$-cuspidal pairs of $\L(\sigma)^F$ and $W_\G(\sigma,(\L,\lambda))^F:=\n_{\G^F_\sigma}(\L)_\lambda/\L^F$.
\end{prop}

\begin{proof}
By definition, $\k^d_{\rm u}(\G^F_\sigma)$ is the number of irreducible characters of the stabiliser $\G^F_\sigma$ with $\ell$-defect $d$ and lying above some unipotent character of $\L(\sigma)^F$. Using Clifford theory and \cite[Theorem 3.2]{Bro-Mal-Mic93} we get
\begin{equation}
\label{eq:Rmk, 1}
\k^d_{\rm u}\left(\G^F_\sigma\right)=\sum\limits_{(\L,\lambda)}\left|\irr^d\left(\G^F_\sigma\enspace \middle|\enspace \E(\L(\sigma)^F, (\L,\lambda))\right)\right|
\end{equation}
where $(\L,\lambda)$ runs over a set of representatives for the action of $\L(\sigma)^F$ on the set of unipotent $e$-cuspidal pairs of $\L(\sigma)^F$. Since \cite[Conjectured Parametrisation B]{Ros24}  holds for $(\G,F)$ with respect to each $(\L,\lambda)$, we deduce that
\begin{equation}
\label{eq:Rmk, 2}
\left|\irr^d\left(\G^F_\sigma\enspace \middle|\enspace \E(\L(\sigma)^F, (\L,\lambda))\right)\right|=\left|\irr^d\left(\n_{\G^F_\sigma}(\L)\enspace \middle|\enspace \lambda\right)\right|.
\end{equation}
On the other hand by \cite[Corollary 2.7]{Spa23}, and applying Clifford's and Gallagher's correspondences (see \cite[Theorem 1.20 and Corollary 1.23]{Nav18}) we know that
\begin{equation}
\label{eq:Rmk, 3}
\left|\irr^d\left(\n_{\G^F_\sigma}(\L)\enspace \middle|\enspace \lambda\right)\right|=\left|\irr^{d-d(\lambda)}\left(\n_{\G^F_\sigma}(\L)_\lambda/\L^F\right)\right|.
\end{equation}
Combining \eqref{eq:Rmk, 1}, \eqref{eq:Rmk, 2}, and \eqref{eq:Rmk, 3} we finally get the equality from the statement.
\end{proof}

As remarked above, in the case of a finite reductive group, the invariants from Definition \ref{def:Defect in Spetses, chains} count the number of irreducible ``unipotent'' characters of the stabiliser $\G_\sigma^F$ with defect $d$. It would be interesting to know if the same can be shown for a $\mathbb{Z}_\ell$-spets $\mathbb{G}$ and a chain $\sigma\in\mathcal{O}(\cT_e(\mathbb{G}))$. More precisely, if there exists a set $\uch(\mathbb{G},\sigma)$ of unipotent characters (analogous to those defined in \cite{Bro-Mal-Mic14}), and an appropriate notion of defect, such that the invariant $\k_{\rm u}^d(\mathbb{G}(q),\sigma)$ from Definition \ref{def:Defect in Spetses, chains} counts the number of characters in $\uch(\mathbb{G},\sigma)$ with defect $d$. Whether or not this is possible, Definition \ref{def:Defect in Spetses, chains} allows us to prove the following analogue of Dade's Conjecture. We refer the reader to \cite[Conjecture C]{Ros24}, \cite[Theorem B]{Ros-Unip}, and the question raised at the end of \cite{Broue22}.

\begin{theo}
\label{thm:Dade-like}
Consider the setting of Proposition \ref{prop:e-HC bijection}. With the notation introduced in Definition \ref{def:Defect in Spetses} and Definition \ref{def:Defect in Spetses, chains}, we have
\[\k^d_{\rm u}(\mathbb{G}(q))-\k^d_{\rm u,c}(\mathbb{G}(q))=\sum\limits_{\sigma\in\mathcal{O}\left(\cT_e^\star(\mathbb{G})\right)}(-1)^{|\sigma|+1}\k^d_{\rm u}(\mathbb{G}(q),\sigma)\]
for every non-negative integer $d$.
\end{theo}

\begin{proof}
Suppose that $\sigma\in\mathcal{O}(\cT_e^\star(\mathbb{G}))$ has final term $\mathbb{L}(\sigma)$ and consider a $\Phi$-cuspidal pair $(\mathbb{L},\lambda)$ of $\mathbb{L}(\sigma)$. If $\mathbb{L}$ is a proper reflection subcoset of $\mathbb{L}(\sigma)$, then we define a new chain $\rho$ by adding $\mathbb{L}$ as the final term of $\sigma$. On the other hand, if $\mathbb{L}$ coincides with $\mathbb{L}(\sigma)$, then we define $\rho$ by removing the term $\mathbb{L}(\sigma)$ from $\sigma$. With this choice of $\rho$, it follows from Definition \ref{def:Relative Weyl group, chains} that $W_\mathbb{G}(\sigma,(\mathbb{L},\lambda))=W_\mathbb{G}(\rho,(\mathbb{L},\lambda))$ and hence
\[\k^d_{\rm u}(\mathbb{G}(q),\sigma, (\mathbb{L},\lambda))=\k^d_{\rm u}(\mathbb{G}(q),\rho, (\mathbb{L},\lambda)).\]
Since $|\sigma|=|\rho|\pm 1$, it follows that the contribution from $\k^d_{\rm u}(\mathbb{G}(q),\sigma, (\mathbb{L},\lambda))$ to $\k^d_{\rm u}(\mathbb{G}(q),\sigma)$ cancels out the contribution of $\k^d_{\rm u}(\mathbb{G}(q),\rho, (\mathbb{L},\lambda))$ to $\k^d_{\rm u}(\mathbb{G}(q),\rho)$. In this way we can cancel out the contribution to the alternating sum coming from all triples $(\sigma,(\mathbb{L},\lambda))$ with the exception of those with $\sigma=\{\mathbb{L}<\mathbb{G}\}=:\sigma_\mathbb{L}$. Therefore, we get
\[\sum\limits_{\sigma\in\mathcal{O}\left(\cT_e^\star(\mathbb{G})\right)}(-1)^{|\sigma|+1}\k^d_{\rm u}(\mathbb{G}(q),\sigma)=\sum\limits_{(\mathbb{L},\lambda)}\k_{\rm u}^d(\mathbb{G}(q),\sigma_\mathbb{L},(\mathbb{L},\lambda))\]
where $(\mathbb{L},\lambda)$ runs over a set of representatives for the $W$-conjugacy classes of $\Phi$-cuspidal pairs of $\mathbb{G}$ such that $\mathbb{L}<\mathbb{G}$. Finally, since $W_\mathbb{G}(\sigma_\mathbb{L},(\mathbb{L},\lambda))=W_\mathbb{G}(\mathbb{L},\lambda)$ by Definition \ref{def:Relative Weyl group, chains}, it follows from Proposition \ref{prop:e-HC bijection} and \eqref{eq:Partition e-HC} that
\[\sum\limits_{(\mathbb{L},\lambda)}\k_{\rm u}^d(\mathbb{G}(q),\sigma_\mathbb{L},(\mathbb{L},\lambda))=\sum\limits_{(\mathbb{L},\lambda)}\k_{\rm u}^d(\mathbb{G}(q),(\mathbb{L},\lambda))=\k_{\rm u}^d(\mathbb{G}(q))-\k_{\rm u, c}^d(\mathbb{G}(q))\]
as required. This completes the proof.
\end{proof}

\section{Appendix}

\subsection{Quillen's theorem A}

Recall that if $\cC$ is a small category, we denote its classifying space by $|\mathcal{C}|$. If $\cC$ and $\cD$ are small categories and $F: \cC \rightarrow \cD$ is a functor then $F$ induces a continuous map $|\cC| \rightarrow |\cD|$ (see, for instance, \cite[Setion 11.2]{Ric20}). Quillen's Theorem A \cite{Qui73}, which we recall below, is a criterion for this map to be a homotopy equivalence.

\begin{defin}[Comma category]
Let $f:\mathcal{C}\to \mathcal{D}$ be a functor between small categories and $y$ an object of $\mathcal{D}$. The \textit{comma category} $f/y$ is the category whose objects are the pairs $(x,\beta)$ where $x$ is an object of $\mathcal{C}$ and $\beta:f(x)\to y$ is a morphism in $\mathcal{D}$. A morphism between two objects $(x,\beta)$ and $(x',\beta')$ in the category $f/y$ is given by a morphism $\alpha:x\to x'$ in $\mathcal{C}$ such that $\beta'\circ f(\alpha)=\beta$.
\end{defin}

We can then state Quillen's criterion as follows.

\begin{theo}[Quillen's Theorem A]
\label{thm:Quillen theorem A}
Let $f:\mathcal{C}\to \mathcal{D}$ be a functor between small categories and suppose that $|f/y|$ is contractible for every object $y$ of $\mathcal{D}$. Then, $f$ induces a homotopy equivalence of classifying spaces $|\mathcal{C}|\simeq|\mathcal{D}|$.
\end{theo}

A useful means of testing the contractibility condition in the hypothesis of Theorem \ref{thm:Quillen theorem A} is provided by Lemma \ref{lem:Filtered categories are contractible}. To state it, we need another definition.

\begin{defin}
\label{def:filtered}
A small category $\mathcal{C}$ is called \textit{filtered} if it is non-empty and
\begin{enumerate}
\item for every pair of objects $x$ and $x'$ of $\mathcal{C}$, there exists an object $x''$ in $\mathcal{C}$ and morphisms $\alpha:x\to x''$ and $\alpha':x'\to x''$; and
\item for every parallel morphisms $\alpha:x\to x'$ and $\alpha':x\to x'$ in $\mathcal{C}$, there exists a morphism $\alpha'':x'\to x''$ such that $\alpha''\circ \alpha=\alpha''\circ\alpha'$.
\end{enumerate} 
\end{defin}

\begin{lem}
\label{lem:Filtered categories are contractible}
If $\mathcal{C}$ is a small filtered category, then $|\mathcal{C}|$ is contractible.
\end{lem}

\begin{proof}
See, for instance, \cite[Proposition 11.3.10]{Ric20}.
\end{proof}

\subsection{The orbit space of a $G$-simplicial complex}
\label{sec:Orbit space, simplicial complex}

Let $\Delta$ be a simplicial complex. We write
\begin{enumerate}
\item $S(\Delta)$ for the poset consisting of the simplices of $\Delta$ ordered by inclusion;
\item $\sd(\Delta):=|\Delta(S(\Delta))|$ for the \textit{barycentric subdivision} of $\Delta$. 
\end{enumerate}

It is well known that $\sd(\Delta) \simeq |\Delta|$, where $|\Delta|$ denotes the geometric realisation of $\Delta$ (see \cite[p.217]{Ben98}). By abuse of notation, we often denote the geometric realisation $|\Delta|$ simply by $\Delta$. If $\Delta$ is, in addition, a $G$-simplicial complex for some finite group $G$ then $\sd(\Delta)$ is a $G$-space and we write 
\[\mathcal{O}_G(\Delta):=\sd(\Delta)/G.\]
for the \textit{$G$-orbit space} of $\Delta$. 

We require a criterion for the contractibility of a certain type of $G$-orbit space due to Bux \cite{Bux99}. This is a reformulation of Morse-theoretic techniques first introduced by Bestvina and Brady \cite{Bes-Bra97}. 

\begin{defin}
\label{d:flag}
Let $\Gamma$ be a directed graph with vertex set $V(\Gamma)$.  The \textit{flag complex} $X(\Gamma)$ of $\Gamma$ is the simplicial complex with vertex set $V(\Gamma)$  and where $\sigma \subseteq V(\Gamma)$  is a simplex if any pair of vertices in $\sigma$ is joined by an edge in $\Gamma$. 
\end{defin}

We require the following additional terminology.

\begin{defin}
\label{d:morse}
Let $\Gamma$ be a directed graph with vertex set $V(\Gamma)$, and $X(\Gamma)$ be the associated flag complex.
\begin{enumerate}
\item  $x \in V(\Gamma)$ is \textit{minimal} if there are no directed edges of form $x\to y$ in $\Gamma$; 
\item A simplex $\sigma$ in $X(\Gamma)$ is \textit{minimal} if it contains a minimal vertex of $\Gamma$;
%\item If $x, y \in V(\Gamma)$, write $x\preceq y$ if there is a directed path from $y$ to $x$ 
\item $\Gamma$ is \textit{well-founded} if there are no directed paths of infinite length;
\item 
If $x \in V(\Gamma)$, write $\Gamma_{\downarrow}^x$ for the subgraph of $\Gamma$ spanned by the endpoints $y$ of edges $x\to y$ in $\Gamma$;
\item If $\sigma$ is a simplex in $X(\Gamma)$ set $\Gamma_{\downarrow}^\sigma:=\displaystyle{\bigcap_{x\in \sigma}\Gamma_{\downarrow}^x}$.
\end{enumerate}
\end{defin}

We can now state Bux's criterion.

\begin{lem}
\label{lem:Morse theory}
Let $G$ be a finite group and $\Gamma$ be a directed $G$-graph (that is $G$ acts on $\Gamma$ by directed graph automorphisms). Suppose that $\Gamma$ is well-founded and that
\begin{enumerate}
\item $G$ acts transitively on the set of minimal vertices of $\Gamma$;
\item for every non-minimal simplex $\sigma\in X(\Gamma)$, the stabiliser $G_\sigma$ acts transitively on the set of minimal elements of $\Gamma_{\downarrow}^\sigma$.
\end{enumerate} 
 Then $X(\Gamma)$ is a $G$-simplicial complex and $\mathcal{O}_G(X(\Gamma))$ is contractible.
\end{lem}

\begin{proof}
This follows from \cite[Theorem 12 and Corollary 13]{Bux99}.
\end{proof}

\subsection{The orbit space of an $EI$-category}
\label{sec:Orbit space, categories}

Here, we follow the presentation given in \cite{Lin05} and \cite{Lin09}. Let  $\mathcal{C}$ be an \textit{$EI$-category}, that is $\cC$ is a small category in which every endomorphism is an automorphism. For a non-negative integer $m$, let $[m]$ be the category with objects $\{0,1, \dots, m\}$ and a morphism from $i$ to $j$ whenever $i\leq j$. The \textit{subdivision category} $S(\mathcal{C})$ of $\mathcal{C}$  is the category whose objects consist of all faithful functors $\sigma:[m]\to \mathcal{C}$ where, for objects $\sigma:[m]\to \mathcal{C}$ and $\tau:[n]\to \mathcal{C}$,
\[\Hom_{S(\cC)}(\sigma,\tau) =\{(\alpha,\mu) \mid \alpha:[m]\to [n] \mbox{ is a functor and } \mu:\sigma\to \tau\circ \alpha \mbox{ is a natural isomorphism }\}\]
(see \cite[Definition 1.1]{Lin05} and \cite[Definition 2.1]{Lin09}).

Note that an object $\sigma$ of $S(\cC)$ can be represented more explicitly as a chain
\begin{center}
\begin{tikzcd}
X_0\arrow[r, "\varphi_0"] &X_1\arrow[r, "\varphi_1"]&\dots\arrow[r, "\varphi_{m-1}"] &X_m
\end{tikzcd}
\end{center}
where each $\varphi_i$ is a non-isomorphism in $\mathcal{C}$. The following property of the subdivision category can be found in \cite[Proposition 1.2]{Lin05}

\begin{lem}
\label{lem:Subdivision is EI}
If $\mathcal{C}$ is an EI-category, then the subdivision category $S(\mathcal{C})$ is an EI-category.
\end{lem}

For any $EI$-category $\cC$, the set  $[\mathcal{C}]$ of isomorphism classes of objects in $\cC$ is naturally endowed with the structure of a poset via the relation
\[[X]\leq [Y] \Longleftrightarrow\Hom_\mathcal{C}(X,Y) \neq \emptyset\]
for objects $X,Y$ in $\cC$. Thus it follows from Lemma \ref{lem:Subdivision is EI} that $[S(\mathcal{C})]$ is a poset and we can consider the following definition.

\begin{defin}
\label{def:orbspace}
Let $\cC$ be an $EI$-category. The \textit{orbit space} $\cO(\cC)$ of $\mathcal{C}$ is the geometric realisation of the order complex of $[S(\mathcal{\cC})]$. That is, $\mathcal{O}(\mathcal{C}):=|\Delta([S(\mathcal{C})])|.$
\end{defin}

\bibliographystyle{alpha}

\vspace{1cm}

{\mbox{\sc{FB Mathematik, RPTU Kaiserslautern--Landau, Postfach 3049, 67653 Kaiserslautern, Germany}}}

\textit{Email address:} \href{mailto:damiano.rossi.math@gmail.com}{damiano.rossi.math@gmail.com}

\vspace{1cm}

{\sc{Department of Mathematical Science, Loughborough University, LE11 3TU, UK}}

\textit{Email address:} \href{mailto:j.p.semeraro@lboro.ac.uk}{j.p.semeraro@lboro.ac.uk}

\end{document}